\documentclass[letterpaper,10pt]{amsart}

\usepackage[all]{xy}                        %

\CompileMatrices                            

\UseTips                                    

\input xypic
\usepackage[bookmarks=true]{hyperref}       

\usepackage{amssymb,latexsym,amsmath,amscd}
\usepackage{xspace}
\usepackage{color}
\usepackage{stmaryrd}
\usepackage{graphicx}

\reversemarginpar

\theoremstyle{plain}
\newtheorem{theorem}{Theorem}[section]
\newtheorem*{theorem*}{Theorem}
\newtheorem{proposition}[theorem]{Proposition}
\newtheorem{corollary}[theorem]{Corollary}
\newtheorem{lemma}[theorem]{Lemma}

\newtheorem{condition}[theorem]{Condition}

\theoremstyle{definition}
\newtheorem{definition}[theorem]{Definition}

\newtheorem{remark}[theorem]{Remark}
\newtheorem{example}[theorem]{Example}

\newcommand{\enm}[1]{\ensuremath{#1}}          %
\newcommand{\op}[1]{\operatorname{#1}}
\newcommand{\cal}[1]{\mathcal{#1}}

\newcommand{\CC}{\enm{\mathbb{C}}}

\newcommand{\ZZ}{\enm{\mathbb{Z}}}

\newcommand{\PP}{\enm{\mathbb{P}}}

\newcommand{\Aa}{\enm{\cal{A}}}

\newcommand{\Ee}{\enm{\cal{E}}}
\newcommand{\Ff}{\enm{\cal{F}}}
\newcommand{\Gg}{\enm{\cal{G}}}
\newcommand{\Hh}{\enm{\cal{H}}}
\newcommand{\Ii}{\enm{\cal{I}}}

\newcommand{\Ll}{\enm{\cal{L}}}

\newcommand{\Oo}{\enm{\cal{O}}}

\newcommand{\Ss}{\enm{\cal{S}}}

\newcommand{\Zz}{\enm{\cal{Z}}}

\renewcommand{\phi}{\varphi}
\renewcommand{\theta}{\vartheta}
\renewcommand{\epsilon}{\varepsilon}

\newcommand{\Pic}{\op{Pic}}

\newcommand{\Hom}{\op{Hom}}
\newcommand{\Ext}{\op{Ext}}

\newcommand{\rk}{\op{rank}}

\renewcommand{\to}[1][]{\xrightarrow{\ #1\ }}

\newcommand{\old}[1]{}

\begin{document}

\title[$2$-nilpotent co-Higgs structures]{$2$-nilpotent co-Higgs structures}
\author{Edoardo Ballico and Sukmoon Huh}
\address{Universit\`a di Trento, 38123 Povo (TN), Italy}
\email{edoardo.ballico@unitn.it}
\address{Sungkyunkwan University, 300 Cheoncheon-dong, Suwon 440-746, Korea}
\email{sukmoonh@skku.edu}
\thanks{The first author is partially supported by MIUR and GNSAGA of INDAM (Italy). The second author is supported by Basic Science Research Program 2015-037157 through NRF funded by MEST and the National Research Foundation of Korea(KRF) 2016R1A5A1008055 grant funded by the Korea government(MSIP)}


\begin{abstract}
A co-Higgs sheaf on a smooth complex projective variety $X$ is a pair of a torsion-free coherent sheaf $\Ee$ and a global section of $\mathcal{E}nd(\Ee)\otimes T_X$ with $T_X$ the tangent bundle. We construct $2$-nilpotent co-Higgs sheaves of rank two for some rational surfaces and of rank three for $\PP^3$, using the Hartshorne-Serre correspondence. Then we investigate the non-existence, especially over projective spaces. 
\end{abstract}

\maketitle

\section{Introduction}
Let $X$ be a smooth projective variety with tangent bundle $T_X$. A co-Higgs bundle, i.e. a pair of an holomorphic bundle $\Ee$ and a morphism $\Phi : \Ee \rightarrow \Ee \otimes T_X$ with $\Phi \wedge \Phi=0$ called the {\it co-Higgs field}, is a generalized holomorphic bundle over a smooth complex projective variety $X$, considered as a generalized complex manifold \cite{Hi, Gual}. It is observed that the existence of a stable co-Higgs bundle gives a constraint on the position of $X$ in the Kodaira spectrum. Indeed, there are no stable co-Higgs bundles with non-zero co-Higgs field on curves of genus at least two, K3 surfaces and surfaces of general type \cite{R2,Rayan}. With the same philosophy,  M.~Corr\^{e}a has shown in \cite{Correa} that the existence of stable co-Higgs bundle of rank two with a non-trivial nilpotent co-Higgs field, forces the base surface to be uniruled up to finite \'etale cover. In \cite{BH} we investigate the surfaces with $H^0(T_X)=H^0(S^2T_X)=0$, which implies that co-Higgs fields are automatically nilpotent. The natural definition of stable co-Higgs bundles allows one to study their moduli and there have been several recent works on the description of the moduli spaces over projective spaces and a smooth quadric surface; see \cite{R1, Rayan, VC}. 

In this article our main concern is the existence and non-existence of a co-Higgs sheaf with a nilpotent co-Higgs field. The Hartshorne-Serre correspondence states that the construction of vector bundles of rank at least two is closely related with the structure of two-codimensional locally complete intersection subschemes. Using the correspondence we produce a nilpotent co-Higgs structure on bundles satisfying a certain condition over various varieties; see Condition \ref{con}. Assuming $\Pic (X) = \ZZ\langle \Oo_X(1)\rangle$ for a very ample line bundle $\Oo_X(1)$, we define $x_{\Ee}$ for a reflexive sheaf $\Ee$ of rank two to be the maximal integer $x$ such that $H^0(\Ee(-x)) \ne 0$, to measure its instability. Then we observe that any nilpotent map associated to $\Ee$ is trivial if $x_{\Ee}$ is low. In case $X=\PP^n$ and rank two, we get the following:

\begin{theorem}\label{ttt1}
The set of nilpotent co-Higgs fields on a fixed stable reflexive sheaf $\Ee$ of rank two on $\PP^n$ is identified with the total space of $\Oo_{\PP^n}(-1)^{\oplus h^0(\Ee(1))}$, with the zero section blown down to a point corresponding to the trivial field, only if $c_1(\Ee)+2x_{\Ee}=-3$. In the other cases the set is trivial.
\end{theorem}

\noindent All co-Higgs structures on $T_{\PP^2}(t)$ are described in \cite[Case $r=2$ of \S 5.5]{R1} and \cite[Theorem 5.9]{Rayan}. In case $X=\PP^3$ we show the existence of some nilpotent co-Higgs structures on some rank three semistable bundles with trivial first Chern class. 

\begin{theorem}\label{ttt2}
For each positive integer $c_2$, there exist both strictly semistable indecomposable bundle and stable bundle of rank three on $\PP^3$ with trivial first Chern class, on which there are nilpotent co-Higgs structures $\Phi$ with $\mathrm{ker}(\Phi)=\Oo_{\PP^3}^{\oplus 2}$. 
\end{theorem}
\noindent We have examples of rank two semistable co-Higgs bundles of several Chern classes on some rational surfaces and the three-dimensional projective space with respect to various polarizations in section 2. In Example \ref{bbb2} we show the existence of semistable co-Higgs bundles of rank two with nilpotent co-Higgs fields over the variety with no global tangent vector fields. In Example \ref{b0} we produce nilpotent co-Higgs structures over a smooth quadric surface and in particular we derive the existence part of \cite[Theorem in page 2]{VC}. 

Then we turn our attention to the non-existence of nilpotent co-Higgs structures. As observed in Lemma \ref{aaaa1}, the existence of non-semistable reflexive sheaf of rank two with semistable co-Higgs structures forces $X$ to be a projective space. From Proposition \ref{d2} any reflexive sheaf of rank two with high stability and extra condition involving new invariant $y_{\Ee}$ turns out to have no non-trivial trace-free co-Higgs structures. So we are driven to focus on projective spaces, especially $\PP^2$ and $\PP^3$. Using Theorem \ref{ttt2} we show the existence of both of strictly semistable indecomposable reflexive sheaf and stable reflexive sheaf of rank two with nilpotent co-Higgs structures for each Chern numbers from the Bogomolov inequality; see Corollaries \ref{zz3} and \ref{zz4}. On the other hand, this existence are not expected to hold for vector bundles due to the following: 

\begin{proposition}\label{ttt3}
If $\Ee$ is a non-splitting and strictly semistable bundle of rank two on $\PP^3$ with the Chern numbers $(c_1, c_2)$ with a non-trivial nilpotent co-Higgs structure, then we have $4c_2-c_1^2 >32$.  
\end{proposition}

\noindent We also get similar result for stable vector bundles of rank two with the condition $4c_2-c_1^2 > 28$; see Proposition \ref{zz7}. In case of $\PP^2$ a general stable rank two bundle has no non-zero trace zero co-Higgs structures, except for very few integers $c_1^2-4c_2$. Indeed, we prove the following result.

\begin{theorem}\label{ttt4}
If $\Ee$ is a general element in the moduli of stable sheaves of rank two on $\PP^2$ with $c_1(\Ee)\in \{-1,0\}$, equipped with a non-trivial trace-free co-Higgs structure, then we have $c_2(\Ee)<5(c_1(\Ee)+5)$. 
\end{theorem}

Then we suggest a condition to insure the non-existence of non-trivial trace-free co-Higgs structure on a reflexive sheaf of rank two on non-projective spaces in Proposition \ref{c2}, using another newly introduced invariant $z_{\Ee}$. 

Let us summarize here the structure of this article. In section $2$ we introduce the definition of semistable co-Higgs sheaves and suggest a condition to construct a nilpotent co-Higgs structure, using the Hartshorne-Serre correspondence. Then we play this construction over several rational surfaces and three-dimensional projective space. In section $3$, we introduce two invariants $x_{\Ee}$ and $y_{\Ee}$ associated to a rank two reflexive sheaf, with which we collect the criterion for the existence and non-existence of non-trivial nilpotent co-Higgs structures. We finish the article in section $4$ by dealing with a criterion of non-existence over non-projective spaces.



\section{Definitions and Examples}

Throughout the article our base field is the field $\CC$ of complex numbers. We will always assume that $X$ is a smooth projective variety of dimension $n$ with tangent bundle $T_X$. For a fixed ample line bundle $\Oo_X(1)$ and a coherent sheaf $\Ee$ on $X$, we denote $\Ee \otimes \Oo_X(t)$ by $\Ee(t)$ for $t\in \ZZ$. The dimension of cohomology group $H^i(X, \Ee)$ is denoted by $h^i(X,\Ee)$ and we will skip $X$ in the notation, if there is no confusion.

\begin{definition}
A {\it{co-Higgs}} sheaf on $X$ is a pair $(\Ee, \Phi)$ where $\Ee$ is a torsion-free coherent sheaf on $X$ and $\Phi \in H^0(\mathcal{E}nd (\Ee) \otimes T_X)$ for which $\Phi\wedge \Phi=0$ as an element of $H^0(\mathcal{E}nd(\Ee) \otimes \wedge^2 T_X)$. Here $\Phi$ is called the {\it{co-Higgs field}} of $(\Ee, \Phi)$ and the condition $\Phi \wedge \Phi=0$ is an integrability condition originating in the work of Simpson \cite{S}. 
\end{definition}

Let $\Ee$ be a torsion-free sheaf on $X$ and $\Phi : \Ee\rightarrow \Ee\otimes T_X$ be a map of $\Oo_X$-sheaves. We say that $\Phi$ is \emph{$2$-nilpotent} if $\Phi$ is non-trivial and $\Phi \circ \Phi =0$. Note that any $2$-nilpotent map $\Phi : \Ee \rightarrow \Ee \otimes T_X$ satisfies $\Phi \wedge \Phi =0$ and so it is a non-zero co-Higgs structure on $\Ee$, i.e. a nilpotent co-Higgs structure.

\begin{condition}\label{con}
For a fixed integer $r\ge2$, a two-codimensional locally complete intersection $Z\subset X$ and $\Aa \in \mathrm{Pic}(X)$, we consider the following two conditions:
\begin{itemize}
\item [(i)] $H^0(T_X\otimes \Aa ^\vee )\ne 0$;
\item [(ii)] the general sheaf fitting into the following exact sequence is locally free, 
\begin{equation}\label{eqb1}
0 \to \Oo _X^{\oplus (r-1)} \stackrel{u}{\to} \Ee \stackrel{v}{\to} \Ii _Z\otimes \Aa \to 0.
\end{equation}
\end{itemize}
\end{condition}

\noindent Our main object of interest is the middle term $\Ee$ in (\ref{eqb1}) with the additional property that it is reflexive. If $X$ is a smooth surface, then reflexivity is equivalent to local-freeness and in the Examples \ref{rt2}, \ref{b3}, \ref{bbb1}, \ref{bbb2}, \ref{bbb3}, \ref{b0} we produce vector bundles. If $n$ is at least $3$, there are many reflexive, but non-locally free sheaves of rank two. In Example \ref{b4} we produce such sheaves.

\begin{remark}
By \cite{wahl} any smooth projective variety $X$ of dimension $n$ satisfying $H^0(T_X(-1)) \ne 0$ is isomorphic to $\PP^n$. So Condition \ref{con}-(i) with $\Aa\cong\Oo_X(1)$ implies that $X=\PP^n$. Note that we always have $H^0(T_X(-2))=0$, except when $X=\PP^1$. 
\end{remark}

\begin{definition}\label{ss1}
For a fixed ample line bundle $\Hh$ on $X$, a co-Higgs sheaf $(\Ee, \Phi)$ is {\it $\Hh$-semistable} (resp. {\it $\Hh$-stable}) if
\begin{align*}
\frac{\det(\Ff)\cdot \Hh^{n-1}}{\rk \Ff} \leq~ (\text{resp.} <)~\frac{\det(\Ee)\cdot \Hh^{n-1}}{\rk \Ee}
\end{align*}
for every coherent subsheaf $0\subsetneq \Ff \subsetneq \Ee$ with $\Phi(\Ff) \subset \Ff \otimes T_X$. In case $\Hh\cong \Oo_X(1)$ we will simply call it semistable (resp. stable) without specifying $\Hh$. 
\end{definition}

\begin{remark}\label{b2}
Take any torsion-free sheaf $\Ee$ fitting into (\ref{eqb1}) with $Z =\emptyset$ and $\Aa$ any numerically trivial line bundle. Then $\Ee$ is $\Hh$-semistable with respect to any polarization $\Hh$. By Lemma \ref{b1}, $\Ee$ has a nonzero $2$-nilpotent co-Higgs field.
\end{remark}

\begin{lemma}\label{b1}
Fix a torsion-free sheaf $\Ee$ fitting into (\ref{eqb1}) and assume Condition \ref{con}-(i). Then there exists a $2$-nilpotent co-Higgs structure on $\Ee$ with $\mathrm{ker}(\Phi )\cong \Oo _X^{\oplus (r-1)}$.
\end{lemma}

\begin{proof}
Any non-zero section $\sigma \in H^0(T_X\otimes \Aa^\vee)$ induces a non-zero map $h: \Ii _Z\otimes \Aa \rightarrow T_X$. Then we may define $\Phi$ to be the following composite:
$$\Ee \stackrel{v}{\to} \Ii_Z\otimes \Aa \stackrel{h}{\to} T_X \stackrel{g}{\to} \Oo_X^{\oplus (r-1)}\otimes T_X \stackrel{u\otimes \mathrm{id}}{\to} \Ee \otimes T_X,$$
where the map $g$ is induced by an inclusion $\Oo_X \rightarrow \Oo_X^{\oplus (r-1)}$. 
\end{proof}

Note that the way of constructing a $2$-nilpotent co-Higgs structure, used in Lemma \ref{b1}, will be used throughout the whole article, specially when we prove the existence of a non-trivial co-Higgs structure. 

\begin{example}\label{rt1}
Take $n= \dim (X)\ge 3$ and assume $H^0(T_X(-D))\ne 0$ for some effective divisor $D$. Lemma \ref{b1} with $\Aa \cong \Oo_X(D)$ gives pairs $(\Ee, \Phi)$, where $\Ee$ is a torsion-free sheaf and $\Phi$ is nonzero with $\Phi\circ \Phi =0$. Note that $(\Ee ,\Phi)$ is stable for any polarization on $X$. We take as $Z$ a smooth two-codimensional subvariety, not necessarily connected. By \cite[Theorem 4.1]{Hartshorne1} it is sufficient that $\omega _Z\otimes \omega _X(D)$ is globally generated. We may take as $Z$ a disjoint union of smooth complete intersections of an element of $|\Oo _X(a)|$ and an element of $|\Oo _X(b)|$ with $\omega _X(a+b)$ globally generated. In particular, there are plenty of non-locally free examples. Among the examples we may take as $X$ the Segre variety
$\PP^{n_1}\times \cdots \times \PP^{n_k}$ with as $D$ the pull-back of $\Oo _{\PP^{n_i}}(1)$ by the projection $\pi _i: X\rightarrow \PP^{n_i}$ on the $i$-th factor.
\end{example}

\begin{example}\label{rt2}
Let $X$ be a smooth and connected projective surface with $H^0(T_X)\ne 0$. Fix an integer $r\ge 2$. In Lemma \ref{b1} we take $\Aa \cong \Oo _X$ and a general subset $Z$ of $X$ with cardinality $s\ge r-1 + h^0(\omega _X)$. Since $Z$ is general and $s>h^0(\omega _X)$, we have $h^0(\omega_X \otimes \Ii _{S\setminus \{p\}})=0$ for each $p\in Z$ and so the Cayley-Bacharach condition is satisfied. Thus the middle term $\Ee$ in the general extension (\ref{eqb1}) is locally free. We have $\det (\Ee) \cong \Oo_X$ and $\Ee$ is strictly semistable for any polarization of $X$. Since $H^0(T_X\otimes \Aa ^\vee )>0$, Lemma \ref{b1} gives the existence of a non-trivial $2$-nilpotent $\Phi : \Ee \rightarrow \Ee\otimes T_X$. From the long exact sequence of cohomology of 
$$0\to \Ii_Z \otimes \omega_X \to \omega_X \to \omega_X \otimes \Oo_Z \to 0,$$
we get $h^1(\Ii_Z \otimes \omega_X) \ge s-h^0(\omega_X) \ge r-1$ and so $\dim \Ext ^1(\Ii _Z,\Oo _X) \ge r-1$. Hence there is $\Ee$ with no trivial factor. Now we check that any locally free $\Ee$ with no trivial factor is indecomposable. Assume $\Ee \cong \Ee _1\oplus \Ee _2$ with $k =\mathrm{rank}(\Ee_1)$ and $1\le k\le r-1$. Let $\Gg _i\subseteq \Ee _i$ for $i=1,2$, be the image of the evaluation map $H^0(\Ee_i)\otimes \Oo _X\rightarrow \Ee_i$. Since $Z$ is not empty, we have $h^0(\Ee)=r-1$ and $\Gg := u(\Oo _X^{\oplus( r-1)})$ is the image of the evaluation map $H^0(\Ee)\otimes \Oo _X\rightarrow \Ee$. Since $\Gg _1\oplus \Gg _2\cong \Gg \cong \Oo _X^{\oplus (r-1)}$ and $\Gg$ is saturated in $\Ee$, $\Gg_i$ is locally free and saturated in $\Ee _i$ for each $i$. Since $\mathrm{rank}(\Gg _1)+\mathrm{rank}(\Gg _2) +1= \mathrm{rank}(\Ee _1)+\mathrm{rank}(\Ee _2)$, there exists $i$ with $\Ee_i = \Gg_i$ and so $\Ee$ has a trivial factor, contradicting our assumptions.
\end{example}

In Condition \ref{con}, if $\Aa$ is negative with respect to a polarization $\Hh$, then the co-Higgs bundle $(\Ee ,\Phi)$ in Lemma \ref{b1} is not $\Hh$-semistable, because $\mathrm{ker}(\Phi )=\Oo _X^{\oplus (r-1)}$. Note that if $\Ee$ is (semi)stable with respect to $\Hh$, then each co-Higgs structure on $\Ff$ has the same property. Thus it is necessary to check when $\Ee$ is (semi)stable and we will focus on the sheaves in Condition \ref{con}-(ii) for a few cases such as
\begin{itemize}
\item blow-ups of $\PP^2$ at a finite set of points;
\item a smooth quadric surface;
\item the three-dimensional projective space $\PP^3$.
\end{itemize}

\begin{example}\label{b3}
Let $X=\PP^2$ and take $\Aa \cong \Oo _{\PP^2}(1)$. Note that the Cayley-Bacharach condition is satisfied for any locally complete intersection zero-dimensional subscheme, or a finite set, $Z$ to get a locally free sheaf $\Ee$ with $c_1(\Ee)=1$ and $c_2(\Ee) =\deg (Z)$. For $\Ee$ to have no trivial factor, it is sufficient to have $\deg (Z) \ge r-1$ and that the extension is general. If $r=2$ and (\ref{eqb1}) does not split, then $\Ee$ is stable. Note that (\ref{eqb1}) does not split if $Z\ne \emptyset$ and $\Ee$ is locally free. 

Now assume $r\ge 3$. Note that $\Ee$ is semistable if and only if it is stable i.e. there is no subsheaf $\Gg \subset \Ee$ with positive degree and rank less than $r$. We assume that $\Ee$ is locally free. Since $h^1(\Oo _{\PP^2})=0$, we have $h^0(\Ee )=r-1+h^0(\Ii _Z(1))$. Assume that $\Ee$ is not semistable and so the existence of a subsheaf $\Gg \subset \Ee$ of rank $s<r$ with maximal positive degree among all subsheaves. Then $\Gg$ is saturated in $\Ee$, i.e. $\Ee /\Gg$ has no torsion. We take $s$ maximal with the previous properties, i.e. if $s\le r-2$ we assume that no subsheaf of $\Ee$ with rank $\{s+1,\dots ,r-1\}$ has positive degree. Since $\deg (\Gg \cap u(\Oo _{\PP^2}^{\oplus (r-1)})) \le 0$, we have $v(\Gg ) \ne 0$ and $\deg (v(\Gg )) >0$. Thus we have $v(\Gg )\cong \Ii _W(1)$ for some zero-dimensional subscheme $W\supseteq Z$. From $h^0(\Ii _W(1)) \le h^0(\Ii _Z(1))$, we get
$$h^0(\Gg )\le h^0(\Ii _Z(1)) +h^0(\Gg \cap u(\Oo _{\PP^2}^{\oplus (r-1)})).$$
Since $v(\Gg )$ is not trivial, we have
$$\left\{
                                           \begin{array}{ll}
                                             \mathrm{rank}(\Gg \cap u(\Oo _{\PP^2}^{\oplus (r-1)}))=s-1, & \hbox{if $s>1$;}\\                                      
                                             \Gg \cap  u(\Oo _{\PP^2}^{\oplus (r-1)})=0, & \hbox{if $s=1$.}
                                            \end{array}
                                         \right.$$
Since $\Gg \cap u(\Oo _{\PP^2}^{\oplus (r-1)})$ is a subsheaf of a trivial vector bundle, we have 
$$\deg (\Gg \cap u(\Oo _{\PP^2}^{\oplus (r-1)}))\le 0 ~,~h^0(\Gg \cap u(\Oo _{\PP^2}^{\oplus (r-1)})) \le s-1.$$ 
It implies that $h^0(\Ee /\Gg )>0$ and so take a trivial subsheaf $\Oo _{\PP^2} \subseteq \Ee /\Gg$. Let $\pi : \Ee \rightarrow \Ee/\Gg$ be the quotient map. If $s\le r-2$, then $\pi^{-1}(\Oo _{\PP^2})$ contradicts the maximality of $s$. Now assume $s=r-1$. Since $\Ee/\Gg$ is a torsion-free sheaf of rank one with a non-zero section, we have $\Ee/\Gg \cong \Oo _{\PP^2}$. Since $h^0(\Ee/\Gg )=1$, we get $h^0(\Gg \cap u(\Oo _{\PP^2}^{\oplus (r-1)}))=s-1 =r-2$. Then any element in 
$$H^0(u(\Oo _{\PP^2}^{\oplus (r-1)}))\setminus H^0(\Gg \cap u(\Oo _{\PP^2}^{\oplus (r-1)}))$$ 
induces a splitting of the surjection $\Ee \rightarrow \Ee/\Gg \cong \Oo _{\PP^2}$, contrary to the assumption that $\Ee$ has no trivial factor.  
\end{example}

\begin{example}\label{bbb1}
Let $\pi: X\rightarrow \PP^2$ be the blow-up at two points, say $p_1$ and $p_2$. Setting $D_i:= \pi^{-1}(p_i)$ for $i=1,2$ and writing
$$\Oo _X(a;0,0):= \pi^\ast \Oo _{\PP^2}(a)~,~ \Oo _X(0;b,0):=\Oo _X(bD_1)~,~\Oo _X(0;0,c):= \Oo _X(cD_2),$$
we have $\omega_X \cong \Oo_X(-3;1,1)$. Let $D\subset X$ be the strict transform of the line through $p_1$ and $p_2$ and then we have $\{D\} = |\Oo _X(1;-1,-1)|$. Recall that for any smooth projective variety $Y$ the vector space $H^0(T_Y)$ is the tangent space
at the identity of the scheme $\mathrm{Aut}(Y)$. So we have $h^0(T_X) =4$, $h^0(T_X(-D_1)) = h^0(T_X(-D_2)) = 6$ and $h^0(T_X(-D_1-D_2)) =h^0(T_X(-D_1-D_2-D)) =4$. Set 
\begin{align*}
\Ss &:= \{\Oo _X,\Oo _X(D_1),\Oo _X(D_2),\Oo _X(D),\Oo _X(D_1+D_2),\Oo _X(D_1+D_2+D)\};\\
\Ss_1&:=\left\{\Oo_X(B_1+B_2-B_3), \Oo_X(B_1-B_2)~|~ \{B_1,B_2,B_3\} = \{D_1,D_2,D\}\right\}.
\end{align*}
If we take as $\Aa$ any element of $\Ss \cup S_1$, then we have $h^0(T_X\otimes \Aa ^\vee )>0$. Note that $h^0(\Aa)=1$ if $\Aa \in \Ss$ and $h^0(\Aa)=0$ if $\Aa \in \Ss_1$. Now fix an integer $r\ge 2$ and take as $Z$ a general subset of $X$ with cardinality $s$ in Condition \ref{con}-(ii). Assume for the moment that the middle term $\Ee$ of (\ref{eqb1}) is locally free. If $\Aa \cong\Oo_X$, then $\Ee$ is strictly semistable for any polarization of $X$. 

Assume $\Aa \in \Ss\setminus \{\Oo _X\}$ and fix a polarization $\Hh$ of $X$. If $\Ll \subset \Ee$ is a saturated subsheaf of rank one with positive $\Hh$-slope, then it is a line bundle. Since $\Ll \cdot \Hh >0$, we have $\Ll \nsubseteq u(\Oo _X)$. Since $\mathrm{Im}(\Phi )\subseteq u(\Oo _X)\otimes T_X$, we have $\Phi (\Ll) \nsubseteq \Ll \otimes T_X$. Thus $(\Ee ,\Phi)$ is $\Hh$-stable.

Now we check a criterion for $s$ with which $\Ee$ is locally free; moreover if $r>2$, we also want $s$ so that $\Ee$ has no trivial factor. In the case $s=0$, $\Ee$ is decomposable and so we may assume $s>0$. First assume $r=2$. In this case we only need to check the Cayley-Bacharach condition. Indeed this condition is satisfied, because $H^0(\omega _X)=0$. Now assume $r>2$ and then by the case $r=2$ a general $\Ee$ fitting into (\ref{eqb1}) is locally free. To check that it has no trivial factor it is sufficient to have $\dim \Ext ^1 (\Ii _Z\otimes \Aa ,\Oo _X) \ge r-1$, because (\ref{eqb1}) is induced by $r-1$ elements of $\Ext ^1 (\Ii _Z\otimes \Aa ,\Oo _X)$ and a trivial factor of $\Ee$ would be a factor of the subsheaf $u(\Oo _X^{\oplus (r-1)})$ of $\Ee$, since we have $h^0(\Ii _Z\otimes \Aa) =0$ due to generality of $Z$. Now for any $\Aa \in \Ss$, we have $\Ext^1(\Ii_Z\otimes \Aa, \Oo_X) \cong H^1(\Ii_Z \otimes \Aa \otimes \omega_X)$ whose dimension is always $s$ and so we may choose $s$ at least $r-1$. 
\end{example}

\begin{example}\label{bbb2}
Let $\pi:X\rightarrow \PP^2$ be the blow-up at three non-collinear points $p_1$, $p_2$ and $p_3$. Set $D_i:= \pi^{-1}(p_i)$ for $i=1,2,3$ and writing $\Oo_X(a;0,0,0):=\pi^\ast \Oo_{\PP^2}(a)$, 
$$\Oo _X(0;b,0,0):=\Oo _X(bD_1)~,~\Oo _X(0;0,c,0):= \Oo _X(cD_2)~,~\Oo _X(0;0,0,d):= \Oo _X(dD_3),$$
we have $\omega _X\cong \Oo _X(-3;1,1,1)$. For any $h\in \{1,2,3\}$, let $T_h\subset X$ be the strict transform of the line through $p_i$ and $p_j$ with $\{h,i,j\} = \{1,2,3\}$. We have $\{T_1\} = |\Oo _X(1;0,-1,-1)|$ and similar formulas hold for $T_2$ and $T_3$. As in Example \ref{bbb1} we have $h^0(T_X) =h^0(T_X(-D_1-D_2-D_3-T_1-T_2-T_3)) =2$.

Let $\Zz$ be the collection of the line bundles $\Oo _X(D)$ with $D>0$ and $D\subseteq D_1\cup D_2\cup D_3\cup T_1\cup T_2\cup T_3$. As in Example \ref{bbb1}, if $\Aa \cong \Oo_X$, then $\Ee$ is stable for any polarization, and if $\Aa \in \Zz$, then $(\Ee, \Phi)$ is stable for any polarization. We may also take as $\Aa$ a line bundle $\Oo _X(B)$ with $B\ne 0$, $B$ a sum of some of the divisors $D_i$ and $T_j$ with sign. In this case $(\Ee ,\Phi)$ is (semi)stable for some polarization, but not for all polarizations. Note that in any case we have $h^0(T_X \otimes \Aa^\vee)>0$. 
\end{example}

\begin{example}\label{bbb3}
Fix an integer $k\ge 3$ and a line $\ell \subset \PP^2$. Let $\pi : X\rightarrow \PP^2$ be the blow-up at $k$ points $p_1,\dots ,p_k\in \ell$. Set $D_i:= \pi^{-1}(p_i)$ for $i=1,\dots ,k$ and let $D\subset X$ be the strict transform of $\ell$. Then we have 
$$(\pi^\ast \Oo _{\PP^2}(1))(-D_1-\cdots -D_k) \cong \Oo _X(D)~,~\omega _X\cong (\pi^\ast \Oo _{\PP^2}(-3))(D_1+\cdots +D_k).$$
We also have $h^0(T_X) =h^0(T_X(-D_1 -\cdots -D_k)) >0$.

Let $\Zz$ be the collection of the line bundles $\Oo _X(T)$ with $T>0$ and $T\subseteq D\cup D_1\cup \cdots \cup D_k$. As in Example \ref{bbb1} and \ref{bbb2}, if $\Aa \cong \Oo_X$, then $\Ee$ is stable for any polarization, and if $\Aa \in \Zz$, then $(\Ee, \Phi)$ is stable for any polarization. We may also take as $\Aa$ a line bundle $\Oo _X(B)$ with $B\ne 0$, $B$ a sum of some of the irreducible components of $D\cup D_1\cup \cdots \cup D_k$ with sign. In this case $(\Ee ,\Phi)$ is (semi)stable for some polarization, but not for all polarizations. Again in any case we have $h^0(T_X \otimes \Aa^\vee)>0$.
\end{example}

\begin{example}\label{b0}
Let $X$ be a smooth quadric surface and take $\Aa$ from 
$$\left\{\Oo _X,\Oo _X(1,0), \Oo _X(2,0) , \Oo _X(0,1), \Oo_X(0,2)\right\}$$
In each case the Cayley-Bacharach condition is satisfied. If $\Aa\cong \Oo _X$, then for any $r\ge 2$ and integer $\deg (Z)\ge 0$ we get vector bundles which are strictly semistable for any polarization (see Example \ref{b3}). Now assume $\Aa \not \cong \Oo _X$ and let $\Hh$ be any polarization on $X$. We claim that $(\Ee,\Phi)$ is $\Hh$-stable. Take an integer $s\in \{1,\dots ,r-1\}$ and a subsheaf $\Gg\subset \Ee$ of rank $s$ with maximal $\Hh$-slope and with $\Phi (\Gg)\subset \Gg\otimes T_X$. We have $\mathrm{Im}(\Phi)\subset \Oo _X\otimes T_X$ and $\mathrm{ker}(\Phi )\cong \Oo _X^{\oplus (r-1)}$. Thus the $\Hh$-slope of $\Gg\cap \mathrm{ker}(\Phi)$ is at most zero. We have $\Phi (\Gg) \subset \Oo _X\otimes T_X$ and so $\Gg \subset \Oo _X^{\oplus (r-1)}$. In particular, we have $\deg _\Hh (\Gg )\le 0$ and so $(\Ee,\Phi)$ is stable for any polarization. In many cases
even $\Ee$ is stable for some or most polarizations.

Assume $\Aa \cong \Oo _X(1,0)$. If $Z=\emptyset$, then $\Ee \cong \Oo _X\oplus \Oo _X(1,0)$ and so $\Ee$ is not semistable for any polarization. Assume $Z\ne \emptyset$ and that $(\Ee ,\Phi)$ is not stable with respect to a polarization $\Hh\cong\Oo_X(a,b)$ with $b<2a$. There is a saturated subsheaf $\Ll = \Oo _X(u,v) \subset \Ee$ of rank one with $av+bu \ge b/2$. In particular, at least one of the integers $u$ and $v$ is positive. Write $\Ee/\Ll \cong \Ii _W(1-u,0-v)$ for some zero-dimensional scheme $W\subset X$. We have $c_2(\Ee) = \deg (W) + v(1-u) -uv$. Composing the inclusion $\Ll \subset \Ee$ with the surjection $v$ in (\ref{eqb1}), we get a non-zero map $f: \Oo _X(u,v)\rightarrow \Ii _Z(1,0)$, and so we get $v\le 0$ and $u=1$. 

First assume $v<0$ and then we have $h^0(\Ll )=0$. Since $H^0(\Ee )\ne 0$, we get $h^0(\Ii _W(0,-v)) >0$. Since $b>0$ and $0<a <2b$, we get $av+b<b/2$, a contradiction. 

Now assume $v=0$ and we get $h^0(\Ll )=2$. Then we have $h^0(\Ee)\ge 2$. Since $Z$ is not empty, (\ref{eqb1}) implies that $Z$ is a single point and so $c_2(\Ee)=1$. From $\Ee/\Ll \cong \Ii _W$, we get $c_2(\Ee ) =\deg (W)$ and so $W$ is a single point. The map $u$ in (\ref{eqb1}) and the inclusion $\Ll \subset \Ee$ induce an injective map $j: \Oo _X(1,0)\oplus \Oo _X\rightarrow \Ee$. Since $j$ is an injective map between vector bundles with the same rank and isomorphic determinant, it is an isomorphism. Thus we have $c_2(\Ee )=0$, a contradiction. The same proof works for the case $\Aa \cong\Oo _X(0,1)$ for the polarization $\Hh\cong\Oo_X(a,b)$ with $a<2b$. 

For $r=2$ we recover most of the existence part in part (1) of \cite[Theorem in page 2]{VC1}. The advantage of the current argument is that we prove stability simultaneously with respect to many polarizations $\Hh \cong\Oo _X(a,b)$ and that we explicitly state that our co-Higgs fields are nilpotent. To be in the framework of part (2) of \cite[Theorem in page 2]{VC1} we need to modify the general set-up. Instead of vector bundles $\Ee$ fitting into the exact sequence (\ref{eqb1}) with $\Aa$ as above, we take vector bundles fitting into the exact sequence with $\Aa\cong\Oo_X(1,-1)$,
\begin{equation}\label{eqb1.1}
0 \to \Oo _X \to \Ee \to \Ii _Z(1,-1)\to 0
\end{equation}
with $Z$ a zero-dimensional scheme, where we have $\det (\Ee )\cong \Oo _X(1,-1)$. By taking a twist by some $\Oo _X(\alpha ,\beta )$ we get vector bundles of rank two with an arbitrary determinant $\Oo _X(\gamma ,\delta )$ with both $\gamma$ and $\delta$ odd. But the twist may destroy the stability with respect to certain polarizations.
\end{example}

\begin{example}\label{b4}
Let $X=\PP^3$ and take $\Aa \cong \Oo _{\PP^3}(1)$. Then we have either
\begin{itemize}
\item $\omega_Z(3)$ is spanned, if $r\ge 3$;
\item $\omega_Z \cong \Oo_Z(-3)$, if $r=2$. 
\end{itemize}
In case of $r =2$, we get curvilinear reflexive sheaves $\Ee$ with $c_2(\Ee)=\deg (Z)$ and $c_3(\Ee) = \deg (\omega _Z) + 3\deg (Z)$; see \cite{B}. We always assume $Z\ne \emptyset$, so that $\Ee$ is indecomposable. We claim that $\Ee$ is stable. Assume the existence of a line bundle $\Oo _X(t)\subset \Ee$ with $t>0$. Composing with the surjection $v:\Ee \rightarrow \Ii _Z(1)$ we get the zero map, because $t>0$ and $Z\ne \emptyset$. Thus we get $\Oo _X(t)\subseteq \Oo _X$, a contradiction.

Now we take $r\ge 3$ and $Z$ a non-empty disjoint union of smooth curves. Assume that $\Ee$ has no trivial factor, e.g. if $Z$ is large, and that $h^0(\Ii _Z(1)) =0$, i.e. $Z$ is not planar.  If $(\Ee,\Phi )$ is not stable, then there is a subsheaf $\Gg \subset \Ee$ of rank $s\in \{1,\dots ,r-1\}$ with $\deg (\Gg )>0$ such that $\Phi (\Gg ) \subset \Gg\otimes T_X$ and $s$ is the minimum among all subsheaves of $\Ee$ with the other properties. Since $\mathrm{Im}(\Phi)\subset T_X$ has rank one, so we get $\mathrm{Im}(\Phi )^{\vee \vee } \cong \Oo _X(1)$, i.e. $\mathrm{Im}(\Phi )\cong \Ii _W(1)$ for some $W\subset \PP^3$ with $\dim (W)\le 1$. $\Gg$ is saturated in $\Ee$, i.e. $\Ee /\Gg$ is torsion-free, and so $\Gg$ is a reflexive sheaf. Since $\Ee$ is assumed to be locally free, in the case $s=1$ we get $\Gg \cong \Oo _X(1)$. We exclude this case, because $\Oo _X(1)\nsubseteq \Ee$. 

Now assume $r=3$ and $s=2$. The map $\Gg \rightarrow \Ii _Z(1)$ induced by the surjection in (\ref{eqb1}) must be non-zero. Due to $s=2$, we get $\Gg \not\cong \Oo _X^{\oplus 2}$ and $\Oo _X^{\oplus 2}$ is the image of the evaluation map $H^0(\Ee) \otimes \Oo _X\rightarrow  \Ee$, we have $h^0(\Gg )\le 1$ and so $h^0(\Ee/\Gg )>0$. Since $\Ee/\Gg$ is a torsion-free sheaf of rank one, we get $\Ee/\Gg \cong \Oo _X$. Since $h^0(\Ee )>h^0(\Gg)$, there is $\sigma \in H^0(\Ee)$ whose image in $\Ee/\Gg \cong \Oo _X$. The map $1\mapsto \sigma$ shows that $\Oo _X$ is a factor of $\Ee$, contradicting our assumption.

Now we assume $r=3$ and list several $Z$ for which the middle term $\Ee$ of a general extension (\ref{eqb1}) with $\Aa \cong \Oo _X(1)$ has not $\Oo _X$ as a factor; in each case we certainly need that $\omega _Z(3)$ is spanned and that $h^0(\omega _Z(3)) \ge 2$. Assume $\Ee \cong \Oo _X\oplus \Gg$. Since $\Ee$ is locally free, so is $\Gg$. Since $h^0(\Gg )=1$ and $h^0(\Gg (-1)) = h^0(\Ee(-1)) =0$, $\Gg$ fits in an exact sequence
\begin{equation}\label{eqb1.2}
0\to \Oo _X \to \Gg \to \Ii _W(1)\to 0,
\end{equation}
where $W$ is a locally complete curve with $\omega _W(3)\cong \Oo _W$ and $h^0(\Ii _W(1)) =0$. We obviously have that $W$ is not reduced. From $H^0(\Gg (-1)) =0$, we get that $\Gg$ is a stable vector bundle of rank two on $\PP^3$ with $c_1(\Gg )=1$ and $c_2(\Gg)=\deg (W)$. The subsheaf $\Oo _X$ of $\Gg$ is the image of the evaluation map $H^0(\Gg )\otimes \Oo _X \rightarrow \Gg$. So the surjective maps in (\ref{eqb1}) and (\ref{eqb1.2}) induce a non-zero map $\Ii _W(1) \rightarrow \Ii _Z(1)$ and so we get $W\supseteq Z$. Since $c_2(\Ff ) =c_2(\Gg)$, we have $\deg (Z)=\deg (W)$ and so $Z=W$, which gives a contradiction each time we chose $Z$ with $\omega _Z\not\cong \Oo _Z(-3)$, e.g. each time we chose as $Z$ a disjoint union of $d$ lines. 
\end{example}


\section{Existence and Non-existence of co-Higgs structures}

Let $X$ be a smooth projective variety of dimension $n$ with $\mathrm{Pic}(X) \cong \ZZ$, where the ample generator $\Oo _X(1)$ is very ample. We keep this assumption until Theorem \ref{ii1}, where we assume $\mathrm{Num}(X)\cong \ZZ$. Set $\delta:=\deg (X)$ with respect to $\Oo_X(1)$. For any reflexive sheaf $\Ee$ of rank two on $X$, define $x_{\Ee}$ to be
\begin{equation}\label{xxe}
\mathrm{max}\{x\in \ZZ~|~ h^0(\Ee(-x))>0\}.
\end{equation}
Then $\Ee$ fits into an exact sequence for a subscheme $Z$ with pure codimension two, 
\begin{equation}\label{eqd1}
0 \to \Oo _X(x_{\Ee})\to \Ee \to \Ii _Z(c_1-x_{\Ee})\to 0,
\end{equation}
where $c_1=c_1(\Ee)$ and $c_2(\Ee) = \deg (Z) +x_{\Ee}(c_1-x_{\Ee})\delta$. Note that we have $h^0(\Ii _Z(c_1-x_{\Ee}-1)) =0$ by definition of $x_{\Ee}$.  

\begin{proposition}\label{cc1}
Let $\Ee$ be a reflexive sheaf of rank two on $X$ with $c_1(\Ee) \in \{-1,0\}$ and $x_{\Ee} \le -2$. Then any nilpotent map $\Phi : \Ee \rightarrow \Ee \otimes T_X$ is trivial.
\end{proposition}

\begin{proof}
If $\Phi \ne 0$, then we have $\mathrm{ker}(\Phi) \cong \Oo _X(t)$ for some $t\le x_{\Ee} \le -2$. Since $\mathrm{Im}(\Phi)$ has rank one with no torsion, we have $\mathrm{Im}(\Phi)\cong \Ii _B(-t+c_1)$ for some closed scheme $B\subset X$ with $\dim (B)\le n-2$. Since $\Omega _X^1(2)$ is globally generated and $\mathrm{Im}(\Phi)$ is a subsheaf of $\Ee\otimes T_X$, we may consider $\mathrm{Im}(\Phi)$ as a subsheaf of $\Ee (2)^{\oplus N}$ for some $N>0$. In particular, we get $-t+c_1-2 \le x_{\Ee}$, a contradiction.
\end{proof}

\begin{proposition}\label{cc1.0}
Assume $X\ne \PP^n$. If $\Ee$ is a reflexive sheaf of rank two on $X$ with $c_1(\Ee )+2x_\Ee =-3$, then any nilpotent map $\Phi : \Ee \rightarrow \Ee \otimes T_X$ is trivial.
\end{proposition}

\begin{proof}
Up to a twist we may assume $c_1(\Ee )=-1$. Assume the existence of a non-zero nilpotent $\Phi : \Ee \rightarrow \Ee \otimes T_X$. We have $\mathrm{ker}(\Phi) \cong
\Oo _X(t)$ for some $t<0$. By Proposition \ref{cc1} we have $t=-1$.  Since $\mathrm{Im}(\Phi)$ has rank one with no torsion, we have $\mathrm{Im}(\Phi)\cong \Ii _B$ for some closed scheme $B\subset X$ with $\dim (B)\le \dim (X)-2$. Since  $\mathrm{Im}(\Phi )\subset \mathrm{ker}(\Phi)\otimes T_X$, we get $H^0(T_X(-1)) \ne 0$, and so
$X=\PP^n$ by \cite{wahl}, a contradiction.
\end{proof}

\begin{remark}\label{cc0}
Let $\Ee$ be a stable reflexive sheaf of rank two on $X$ with $c_1(\Ee) =-1$. By the stability of $\Ee$, we have $x_{\Ee}\le -1$. If $x_{\Ee}\le -2$, then any nilpotent map $\Phi :\Ee \rightarrow \Ee \otimes T_X$ is trivial by Proposition \ref{cc1}. As an example, we may take as $\Ee$ the Horrocks-Mumford bundle; $X=\PP^4$, $c_1=-1$ and $c_2=4$. If $x_{\Ee}=-1$, then $\Ee$ fits in an exact sequence
$$0 \to \Oo _X\to \Ee (1)\to \Ii _Z(1)\to 0$$
for some $2$-codimensional scheme $Z\subset X$. Assume $H^0(T_X(-1)) \ne 0$ and so $X=\PP^n$ by \cite{wahl}. Then by Lemma \ref{b1} there exists a non-trivial nilpotent map $\Phi : \Ee \rightarrow \Ee \otimes T_{\PP^n}$ with $\ker (\Phi)=\Oo_{\PP^n}$. 
\end{remark}

\begin{proposition}\label{cc2}
Let $\Ee$ be a stable reflexive sheaf of rank two on $\PP^n$ with $c_1(\Ee)=0$. Then there exists no non-trivial nilpotent map $\Phi :\Ee \rightarrow \Ee \otimes T_{\PP^n}$.
\end{proposition}

\begin{proof}
Since $\Ee$ is stable, we have $\mathrm{ker}(\Phi) \cong \Oo _{\PP^n}(t)$ for some $t\le -1$ and the proof of Proposition \ref{cc1} gives $t=-1$. Since $\mathrm{Im}(\Phi)$ has rank one with no torsion, we have $\mathrm{Im}(\Phi)\cong \Ii _B(1)$ for some closed subscheme $B\subsetneq \PP^n$.

First assume $\dim B\le n-2$. Since $\Phi \circ \Phi =0$, we have $\mathrm{Im}(\Phi) \subset \mathrm{ker}(\Phi)\otimes T_{\PP^n} \cong T_{\PP^n}(-1)$. In particular, we get a nonzero map $h: \Ii _B(1)\rightarrow T_{\PP^n}(-1)$. Since $T_{\PP^n}(-2)$ is locally free and $\dim B \le n-2$, we have 
$$H^0(\PP^n\setminus B,T_{\PP^n}(-2))_{|\PP^n\setminus B}) = H^0(\PP^n,T_{\PP^n}(-2))$$
by \cite[Proposition 1.6]{Hartshorne1}, which is trivial. But the map $h$ gives $H^0(T_{\PP^n}(-2)) \ne 0$, a contradiction. 

Now assume that $B$ contains a hypersurface of degree $e$. We get $\mathrm{Im}(\Phi)\cong \Ii _Z(1-e)$ for some closed subscheme $Z$ with $\dim Z \le n-2$. Since $c_1(\Ee )=0$ and $e>0$, $\Ee$ is not stable, a contradiction.
\end{proof}

\begin{proof}[Proof of Theorem \ref{ttt1}:]
Denote by $\Ss$ the set of all nilpotent maps and up to a twist we may assume $c_1(\Ee)\in \{-1,0\}$. By Proposition \ref{cc2} we can consider only the case of $c_1(\Ee )=-1$. By Proposition \ref{cc1} we have $\Ss =\{0\}$, unless $x_{\Ee}=-1$.
Thus we may assume $x_{\Ee} =-1$ and so $\Ee$ fits into an exact sequence
\begin{equation}\label{eqcc2}
0 \to \Oo _{\PP^n}(-1) \stackrel{\sigma}{\to} \Ee \to \Ii _Z \to 0
\end{equation}
with $Z$ of codimension $2$. By Lemma \ref{b1} a cheap way to get a non-trivial $\Phi$ is to take the composition of the surjection in (\ref{eqcc2}) with the inclusion $\Ii _Z\rightarrow \Oo _{\PP^n}(-1)\otimes T_{\PP^n}$. In this way we get an $(n+1)$-dimensional vector space contained in $\Ss$, isomorphic to $H^0(T_{\PP^n}(-1))$. Conversely, choose any arbitrary nonzero map $\Phi \in \Ss$. The proof of Proposition \ref{cc1} gives $\mathrm{ker}(\Phi) \cong \Oo _{\PP^n}(-1)$ and so $\mathrm{Im}(\Phi)\cong \Ii _B$ for some closed subscheme $B\subsetneq \PP^n$ of codimension two. Since $\Phi \circ \Phi =0$, we have $\mathrm{Im}(\Phi) \subset \mathrm{ker}(\Phi)\otimes T_{\PP^n} \cong T_{\PP^n}(-1)$, and thus $\Phi$ is also obtain by the same way as in Lemma \ref{b1}. Thus any such nilpotent map is represented by an element in $H^0(\Ee(1)) \times H^0(T_{\PP^n}(-1))$ with an action of $\CC^*$ defined by $c\cdot (\sigma, s)=(c\sigma, c^{-1}s)$. Thus the set of nilpotent maps is parametrized by 
$$H^0(\Ee(1))\times H^0(T_{\PP^n}(-1)) \sslash \CC^*,$$
which is the total space of $\Oo_{\PP^n}(-1)^{\oplus a}$ with $a=h^0(\Ee(1))$. Now the assertion follows from the observation that non-proportional sections of $\Ee(1)$ have different zeros as in \cite[Theorem 4.1]{Hartshorne1} and that if $\sigma$ of $s$ is trivial, then the pair $(\sigma, s)$ corresponds to the trivial nilpotent map.  
\end{proof}

We still assume that $X$ is a smooth projective variety with $\Pic(X)\cong \ZZ$ generated by an ample line bundle $\Oo _X(1)$ and $H^0(T_X(-2))=0$, which excludes the case $X\cong \PP^1$ by \cite{wahl}. Let $\Ee$ be a non-semistable reflexive sheaf of rank two on $X$ such that $(\Ee, \Phi)$ is semistable for a map $\Phi : \Ee \rightarrow \Ee \otimes T_X$. Without loss of generality we assume that $\Ee$ is initialized, i.e. $H^0(\Ee )\ne 0$ and $H^0(\Ee (-1)) =0$. Since $\Ee$ is not semistable, we have an exact sequence   
\begin{equation}\label{eqss1}
0 \to \Oo _X\to \Ee \to \Ii _Z(-b)\to 0
\end{equation}
with $b >0$ and $\dim (Z)\le \dim (X)-2$. 

\begin{lemma}\label{aaaa1}
Let $\Ee$ be a non-semistable reflexive sheaf of rank two on $X$ with $(\Ee, \Phi)$ semistable. Then we have $X\cong \PP^n$ with $n\ge 2$ and $b=1$. Also we have either 
\begin{itemize}
\item $\Ee \cong \Oo _{\PP^n}\oplus \Oo _{\PP^n}(-1)$, or 
\item $n=2$ and $Z$ is a point of $\PP^2$.
\end{itemize}
\end{lemma}

\begin{proof}
Since $\Ee$ is reflexive, either $Z=\emptyset$ or $Z$ has pure codimension $2$. From (\ref{eqss1}) we get an exact sequence
\begin{equation}\label{eqss2}
0 \to \Oo _X\otimes T_X \to \Ee\otimes T_X \stackrel{v}{\to} \Ii _Z\otimes T_X(-b)\to 0.
\end{equation}
Since $(\Ee ,\Phi)$ is semistable, we have $\Phi (\Oo _X)\nsubseteq \Oo _X\otimes T_X$ and so $v\circ \Phi : \Oo _X\rightarrow \Ii _Z\otimes T_X(-b)$ is a non-zero map. Since $X\not\cong \PP^1$ by \cite{wahl}, we have $X\cong \PP^n$ with $n\ge 2$ and $b=1$. We also get $h^0(\Ii _Z\otimes T_X)>0$. The zero-locus of each non-zero section of $T_X(-1)$ is a single point. Hence we have either $Z=\emptyset$, or $n=2$ and $Z$ is a single point. If $Z =\emptyset$, then (\ref{eqss1}) gives $\Ee \cong \Oo _{\PP^n}\oplus \Oo _{\PP^n}(-1)$.
\end{proof}

Recall that $x_{\Ee}$ depends only on the isomorphism class of $\Ee$; see (\ref{xxe}). For any $\Ee$ fitting into (\ref{eqd1}) with $Z$ satisfying $h^0(\Ii _Z(c_1-x_{\Ee}-1)) =0$, we know that $\Ee$ is stable (resp. semistable) if and only if $2x_{\Ee}<c_1$ (resp. $2x_{\Ee} \le c_1$). For a fixed $\Ee$, the same subscheme $Z\subset X$ may occur only by proportional sections in $H^0(\Ee (-x))$ by \cite[Proposition 1.3]{hart}. Define $y_{\Ee}$ to be 
$$\mathrm{min}\{ y\ge 0~|~h^0(\Ii _Z(c_1-x_{\Ee}+y)) >0\}.$$
Note that $y_{\Ee} =0$ if and only if $\Ee$ has at least two non-proportional maps $\Oo _X(x)\rightarrow \Ee$ and so fits in at least two non-proportional sequences (\ref{eqd1}), with different subschemes $Z$. Thus in all cases the integer $y_{\Ee}$ is well-defined.

\begin{lemma}\label{d1}
Let $\Ee$ be a reflexive sheaf of rank two on $X$ with $c_1-2x_{\Ee} >0$. Then we have $h^0(\mathcal{E}nd (\Ee)(z)) = h^0(\Oo _X(z))$ for $0\le z < \min \{x_{\Ee} +y_{\Ee}, c_1-2x_{\Ee}\}$.
\end{lemma}

\begin{proof}
Set $x:= x_{\Ee}$ and $y:=y_{\Ee}$, and assume that $\Ee$ fits in (\ref{eqd1}) for some $Z$. Fix $f\in \Hom (\Ee ,\Ee (z))$ and let $f_1: \Ee \rightarrow \Ii _Z(c_1-x+z)$ be the map obtained by composing $f$ with the map $ \Ee(z) \rightarrow \Ii _Z(c_1-x+z)$ twisted from (\ref{eqd1}) with $\Oo _X(z)$. From the assumption $z< x+y$, we have $f_1(\Oo _X(x)) =0$ and so $f$ induces $f_2: \Ii _Z(c_1-x) \rightarrow \Ii _Z(c_1-x+z)$. Now take $g\in H^0(\Oo _X(z))$ inducing $f_2$ and let $\gamma : \Ee \rightarrow \Ee (z)$ be obtained by the multiplication by $g$. Our claim is that $f =\gamma$. Taking $f-\gamma$ instead of $f$ we reduce to the case $g=0$ and in this case we need to prove that $f=0$, when we have $f(\Ee )\subseteq \Oo _X(x+z)$. Since $\Ee$ is reflexive of rank two, we have $\Ee ^\vee \cong \Ee (-c_1)$. Thus $f: \Ee \rightarrow \Oo _X(x+z)$ is induced by a unique $a\in H^0(\Ee (x+z-c_1))$. Since $x+z-c_1< - x$, we have $a=0$ and so $f=0$.
\end{proof}

\begin{proposition}\label{d2}
If $\Ee$ is a reflexive sheaf of rank two on $X$ with 
$$\mathrm{min}\{x_{\Ee} +y_{\Ee}, c_1(\Ee )-2x_{\Ee} \}\ge 3,$$ 
then it has no non-zero trace-free co-Higgs field, not even a non-integrable one.
\end{proposition}

\begin{proof}
Take any map $\Phi : \Ee \rightarrow \Ee \otimes T_X$. Since $\Oo _X(1)$ is very ample, $\Omega ^1_X(2)$ is spanned and so $T_X$ is a subsheaf of $\Oo _X(2)^{\oplus N}$, where $N=h^0(\Omega_X^1(2))$. Thus $\Phi$ induces $N$ elements $\Phi _i: \Ee \rightarrow \Ee (2)$ with $i=1,\dots ,N$. By Lemma \ref{d1} each $\Phi _i$ is induced by $f_i\in H^0(\Oo _X(2))$. Composing the trace map $\mathrm{Tr}(\Phi ): \Oo _X\rightarrow T_X$ of $\Phi$ with the inclusion $T_X\subset \Oo _X(2)^{\oplus N}$, we also get $N$ elements $g_i\in H^0(\Oo _X(2))$. Note that we have $2f_i=g_i$ for all $i$. If $\Phi$ is trace-free, then we get $g_i=0$ and so $f_i=0$ for all $i$. Thus $\Phi$ is trivial.
\end{proof}


\subsection{Case $X=\PP^2$}
For $(c_1,c_2)\in \ZZ^{\oplus 2}$, let $\mathbf{M}_{\PP^2}(c_1,c_2)$ denote the moduli space of stable vector bundles of rank two on $\PP^2$ with Chern numbers $(c_1,c_2)$. Schwarzenberger proved that $\mathbf{M}_{\PP^2}(c_1,c_2)$ is non-empty if and only if $ -4\ne c_1^2-4c_2 < 0$; see \cite[Lemma 3.2]{hart}. When non-empty, $\mathbf{M}_{\PP^2}(c_1,c_2)$ is irreducible; see \cite{maru,barth,hulek,lepo}. For $\Ee\in \mathbf{M}_{\PP^2}(c_1,c_2)$ and any $t\in \ZZ$, we have 
\begin{align*}
c_1(\Ee (t)) &=c_1+2t,\\
c_2(\Ee (t)) &=c_2+t^2+tc_1, \\
\chi (\Ee (t)) &= (c_1+2t+2)(c_1+2t+1)/2 +1 -c_2-t^2-t(c_1+t)\\
                   &= c_1(c_1+2t+3)/2 +(t+1)(t+2)-c_2;
\end{align*}
see \cite[page 469]{b}. Up to a twist we may assume that $c_1\in \{-1,0\}$. Since $\Ee$ is stable, we have $h^0(\Ee )=0$ and so $x_{\Ee} <0$. Define an integer $\alpha(c_1, c_2)$ as 
$$\alpha (c_1,c_2):=\mathrm{min}\{t\in \ZZ_{>0}~|~c_1(c_1+2t+3)/2 +(t+1)(t+2) >c_2\}.$$
For any $\Ee \in \mathbf{M}_{\PP^2}(c_1,c_2)$, we have $\chi (\Ee (a)) >0$ for all $a\ge \alpha (c_1,c_2)$, and $\alpha (c_1,c_2)$ is the minimal positive integer with this property; see \cite[Proposition 7.1]{hart}. By \cite[Theorem 5.1]{b}, a general bundle $\Ee \in \mathbf{M}_{\PP^2}(c_1,c_2)$ has $x_{\Ee} = -\alpha (c_1,c_2)$ and $h^1(\Ee (t)) =0$ for all $t\ge \alpha (c_1,c_2)$. By Proposition \ref{cc2}, if $c_1$ is even, no bundle $\Ee\in \mathbf{M}_{\PP^2}(c_1,c_2)$ has a non-zero nilpotent map $\Phi : \Ee \rightarrow \Ee \otimes T_X$. If $c_1$ is odd, we have the following.

\begin{proposition}\label{++}
Let $\Ee$ be s general element of $\mathbf{M}_{\PP^2}(-1,c_2)$ with $c_2\ge 4$. If $\Phi : \Ee \rightarrow \Ee \otimes T_{\PP^2}$ is a nilpotent map, then we have $\Phi =0$.
\end{proposition}

\begin{proof}
By Proposition \ref{cc1} it is sufficient to prove that $x_\Ee \le -2$, i.e. $h^0(\Ee (1)) =0$. Note that $\chi (\Ee (1)) = 4-c_2 \le 0$ and so we may apply \cite[Theorem 5.1]{b}.
\end{proof}

For any $x\in \ZZ$, let $\mathbf{M}_{\PP^2}(c_1,c_2,x)$ denote the set of all $\Ee \in \mathbf{M}_{\PP^2}(c_1,c_2)$ with $x_\Ee =x$. It is an irreducible family and we have a description of the nilpotent co-Higgs fields on each bundle in $\mathbf{M}_{\PP^2}(c_1, c_2, x)$; see Theorem \ref{ttt1}.

\begin{remark}\label{+++}
Any $\Ee \in \mathbf{M}_{\PP^2}(-1,c_2)$ with $\Ee(-1)$ as in Lemma \ref{b1} and (\ref{eqb1}) for $\Aa \cong \Oo_{\PP^2}(1)$ occurs in an exact sequence
\begin{equation}\label{eqboo1}
0 \to \Oo _{\PP^2}(-1)\to \Ee \to \Ii _Z \to 0
\end{equation}
with $Z$ a locally complete intersection scheme $Z\subset \PP^2$ with $\deg (Z)=c_2(\Ee)$, using that $\deg (Z) =c_2(\Ee (1))$ by \cite[Corollary 2.2]{Hartshorne1}. Since $Z$ is not empty, every vector bundle fitting into (\ref{eqboo1}) is stable and so we have $x_\Ee =-1$. The general element of $\mathbf{M}_{\PP^2}(-1,c_2,-1)$ admits an extension (\ref{eqboo1}) with as $Z$ the general
subset of $\PP^2$ with cardinality $c_2$.
\end{remark}

For a general stable vector bundle of rank two on $\PP^2$, we have $y_{\Ee}\le 1$ by \cite{b} and so we cannot use Proposition \ref{d2} for it. We prove Theorem \ref{ttt4} using the following key observation.

\begin{remark}\label{d3}
Take an irreducible family $\Gamma$ of reflexive sheaves of rank two on $X$. Let $\Gg$ denote the general element of $\Gamma$. Assume the existence of some $\Ee\in \Gamma$ with $c_1(\Ee )-2x_{\Ee} \ge 3$ and $y_{\Ee} +x_{\Ee} \ge 3$. By Lemma \ref{d1} we have $h^0(\mathcal{E}nd (\Ee )(2)) =h^0(\Oo _X(2))$, which is the minimum possibility for $h^0(\mathcal{E}nd(\Gg)(2))$ with $\Gg$ reflexive of rank two on $X$, i.e. $H^0(\mathcal{E}nd (\Ee )(2))$ has the minimal dimension among all reflexive sheaves of rank two on $X$. By the semicontinuity theorem we have $h^0(\mathcal{E}nd (\Gg )(2)) =h^0(\Oo _X(2))$. Thus we may apply the proof of Proposition \ref{d2} to $\Gg$, even when $\Gg$ does not satisfy the assumptions of Proposition \ref{d2}.
\end{remark}

\begin{proof}[Proof of Theorem \ref{ttt4}:]
The proof of Proposition \ref{d2} shows that it is enough to prove $h^0(\mathcal{E}nd (\Ee )(2)) =6$. And by semicontinuity it is also sufficient to prove that $h^0(\mathcal{E}nd (\Gg )(2)) =6$ for some $\Gg \in \mathbf{M}_{\PP^2}(c_1,c_2)$. Furthermore, by Lemma \ref{d1} it is sufficient to find $\Gg \in \mathbf{M}_{\PP^2}(c_1,c_2)$ with $x_\Gg = -2$ and $y_\Gg \ge 5$. 

Now take a general $S\subset \PP^2$ with $\sharp (S) = c_2 +4 +2c_1$ and let $\Gg$ be a general sheaf fitting into
$$0 \to \Oo _{\PP^2}(-2) \to \Gg \to \Ii _S(c_1+2)\to 0.$$
By Bogomolov inequality we have $4c_2>c_1^2$. We have $h^0(\Ii _{S\setminus \{p\}}(c_1+1)) =0$ for $p\in S$ and so the Cayley-Bacharach condition is satisfied. Thus $\Gg$ is locally free. We also have $h^0(\Gg(1))=0$ from $h^0(\Ii _S(c_1+1)) =0$, and so we have $x_\Gg =-2$. On the other hand, we have $\sharp (S) > \binom{c_1+8}{2}$, we have $h^0(\Ii _S(c_1+6))=0$ and so $y_\Gg \ge 5$. Now we may use Remark \ref{d3} and the irreducibility of $\mathbf{M}_{\PP^2}(c_1,c_2)$.
\end{proof}


\subsection{Case $X=\PP^3$ and $r\ge 3$}
We look at locally free sheaves $\Ee$ of rank at least three on $\PP^3$ fitting into (\ref{eqb1}) with either $\Aa \cong \Oo _{\PP^3}$ or $\Aa \cong \Oo _{\PP^3}(1)$. By Lemma \ref{b1} any such a sheaf $\Ee$ has a $2$-nilpotent $\Phi : \Ee\rightarrow \Ee \otimes T_{\PP^3}$ with $\mathrm{ker}(\Phi) \cong \Oo _{\PP^3}^{\oplus (r-1)}$. If $\Aa \cong \Oo _{\PP^3}$, then any torsion-free $\Ee$ fitting into (\ref{eqb1}) is strictly slope-semistable. Note also that if $Z$ is empty in (\ref{eqb1}),
then $\Ee\cong \Oo _{\PP^3}^{\oplus (r-1)}\oplus \Aa$ and that $\deg (Z) =c_2(\Ee)$. In particular if $\Ee$ is not a direct sum of line bundles, then we have $c_2(\Ee)>0$.

\begin{lemma}\label{zz0}
Let $\Ee$ be a reflexive sheaf of rank three fitting into (\ref{eqb1}) with $\Aa \cong \Oo _{\PP^3}(1)$. Then the followings are equivalent. 
\begin{itemize}
\item [(i)] $\Ee$ is slope-semistable;
\item [(ii)] $\Ee$ is slope-stable;
\item [(iii)] $\Ee$ has no trivial factor.
\end{itemize}
\end{lemma}

\begin{proof}
Assume that $\Ee$ has a saturated subsheaf $\Gg$ of rank $s<3$ with $\deg (\Gg )/s\ge 1/3$. 

If $s=1$, then we have $\Gg \cong \Oo _{\PP^3}(t)$ for some $t>0$, because $\Ee$ is reflexive and $\Ee/\Gg$ has no torsion (see \cite[Propositions 1.1 and 1.9]{Hartshorne1}). Then we have $\Gg \nsubseteq u(\Oo _{\PP^3}^{\oplus 2})$ and so $v(\Oo _{\PP^3}(t))$ is a non-zero subsheaf of $\Ii _Z(1)$. In particular, we get $t=1$ and $Z=\emptyset$. Thus we have $\Ee \cong \Oo _{\PP^3}(1)\oplus \Oo _{\PP^3}^{\oplus 2}$. 

Now assume $s=2$. Again $v(\Gg )$ is a non-zero subsheaf of $\Ii _Z(1)$ and so we get $\deg (\Gg) =1$ and that $\Gg$ is an extension of some $\Ii _W(1)$ with $W\supseteq Z$ by $\Oo _{\PP^3}$. It implies that the torsion-free sheaf $\Ee/\Gg$ is a rank one sheaf of degree zero with $h^0(\Ee /\Gg )>0$. Thus we have $\Ee/\Gg \cong \Oo _{\PP^3}$ and the map $u(\Oo _{\PP^3}^{\oplus 2}) \rightarrow \Ee/\Gg$ is surjective. Taking a section of  $u(\Oo _{\PP^3}^{\oplus 2})$ with $1\in H^0(\Ee/\Gg)$ as its image, we get a map $\Ee/\Gg \rightarrow \Ee$ inducing a splitting $\Ee \cong \Gg \oplus \Oo _{\PP^3}$. Thus (iii) implies (ii). Clearly (ii) implies (i). Now assume that $\Ee$ has a trivial factor, i.e. $\Ee \cong \Oo_{\PP^3}\oplus \Ff$ with $\Ff$ a bundle of rank two. Then the slope of $\Ff$ is $1/2$, which is greater than the slope of $\Ee$. Thus (i) implies (iii). 
\end{proof}

\begin{proof}[Proof of Theorem \ref{ttt2}:]
For the strictly semistable bundle, we apply Lemma \ref{b1} with $\Aa \cong \Oo_{\PP^3}$. Except the indecomposability, it is sufficient to find a locally Cohen-Macaulay curve $Z\subset \PP^3$ of $\deg (Z)=c_2$ such that $\omega _Z(4)$ is spanned and there is a $2$-dimensional linear subspace $V\subseteq H^0(\omega _Z(4))$ spanning $\omega_Z(4)$ at each point of $Z_{\mathrm{red}}$. We may even take a smooth $Z$. Note that for every smooth and connected curve $Z\subset \PP^3$, $\omega _Z(4)$ is spanned and non-trivial, and so we get $h^0(\omega _Z(4)) \ge 2$. Since $\omega _Z(4)$ is a line bundle on a curve $Z$, a general $2$-dimensional linear subspace of $H^0(\omega _Z(4))$ spans $\omega _Z(4)$. 

Assume now that $\Ee$ is decomposable. Using the same argument in the proof of Lemma \ref{zz0} to show that (iii) implies (ii), we get $\Ee\cong \Oo _{\PP^3}\oplus \Gg$ for some vector bundle $\Gg$ of rank two. Since $Z$ is not empty, we have $\Gg \not\cong \Oo _{\PP^3}^{\oplus 2}$ and we see that $\Gg$ fits in (\ref{eqb1}) with the same $Z$ above and $r=2$. Thus we get $\omega _Z\cong \Oo _Z(-4)$ by \cite[Theorem 1.1]{hart}, contradicting the assumption that $Z$ is a reduced curve.

For the stable bundle, we follow the argument above with $\omega _Z(3)$ instead of $\omega _Z(4)$.
\end{proof}

For any reflexive sheaf $\Gg$ of rank two on $\PP^3$ we have $c_1(\Gg (t))^2-4c_2(\Gg (t)) = c_1(\Gg )^2-4c_2(\Gg)$ for all $t\in \ZZ$. Take $\Ee$ produced by Theorem \ref{ttt2} and consider its quotient by a subsheaf $\Oo _X\subset \Ee$, or use (\ref{eqb1}) for $r=2$ and the Hartshorne-Serre correspondence in \cite[Thorem 4.1]{Hartshorne1}. Then we get the following results (for the ``~only if~'' part use \cite[Corollary 3.3]{Hartshorne1}).

\begin{corollary}\label{zz3}
For a fixed pair of integers $(c_1, c_2)$ with $c_1$ even, there are an indecomposable and strictly semistable reflexive sheaf $\Ee$ of rank two on $\PP^3$ with $c_1(\Ee )=c_1$ and $c_2(\Ee ) =c_2$, and a non-trivial nilpotent map $\Phi : \Ee \rightarrow \Ee  \otimes T_{\PP^3}$ if and only if $c_1^2-4c_2 <0$.
\end{corollary}

\begin{corollary}\label{zz4}
For a fixed pair of integers $(c_1, c_2)$ with $c_1$ odd, there is a stable reflexive sheaf $\Ee$ of rank two on $\PP^3$ with $c_1(\Ee )=c_1$ and $c_2(\Ee ) =c_2$, equipped with a non-trivial nilpotent map $\Phi : \Ee \rightarrow \Ee\otimes T_{\PP^3}$ if and only if $c_1^2-4c_2 <0$.
\end{corollary}

\begin{remark}\label{zz5}
The interested reader may state and prove statements similar to Theorem \ref{ttt2} and Corollary \ref{zz3} that involve Lemma \ref{b1} with $\Aa \cong \Oo _{\PP^3}$, when $X$ is the three-dimensional smooth quadric $Q_3\subset \PP^4$, using $\omega _Z(-3)$ instead of $\omega _Z(-4)$. 
\end{remark}

\begin{proof}[Proof of Proposition \ref{ttt3}:]
Since $4c_2(\Ee (t)) -c_1(\Ee (t))^2 $ is a constant function on $t$, we may reduce to the case $c_1=0$. By \cite{gotz1}, \cite{gotz2} and \cite[Appendix C]{ik}, we see that $\Ee$ must be as in Lemma \ref{b1} and (\ref{eqb1}) with $r=2$ and $\Aa \cong\Oo _{\PP^3}$. Since we have $d:=\deg (Z)=c_2(\Ee )$, so we get $c_2(\Ee )=0$ if and only if $Z=\emptyset$, i.e. $\Ee \cong \Oo _{\PP^3}^{\oplus 2}$. For the conclusion, it is sufficient to exclude the Chern numbers $c_2$ with $1\le c_2\le 8$. If such $\Ee$ exists, then $Z$ is a locally complete intersection and $\omega _Z \cong \Oo _Z(-4)$. By the duality we have $2\chi (\Oo _Z) = \deg (\omega _Z) = -4d$, i.e. $\chi (\Oo _Z) =-2d$. 

Macaulay proved that a polynomial $q(t)$ is the Hilbert function of a curve of degree $d$ in some $\PP^n$, not necessarily locally a complete intersection, if and only if there is a non-negative integer $\alpha$ such that 
$$q(t) = \sum _{i=0}^{d-1} (t+i-i) +\alpha= dt -(d-2)(d-3)/2+\alpha;$$
see \cite{gotz1, gotz2, ik}; for locally Cohen-Macaulay space curves, one can also use \cite{beo}. If $p(t)$ is the Hilbert polynomial of the scheme $Z$, then we have $\chi (\Oo _Z)=p(0)$ and so $-(d-2)(d-3)/2 \le -2d$, i.e. $(d-2)(d-3)\ge 4d$. But it is false if $1\le d \le 8$.
\end{proof}

\begin{proposition}\label{zz7}
For a fixed pair of integers $(c_1, c_2)$ with $c_1$ odd and $4c_2-c_1^2 \le 28$, there is no pair $(\Ee ,\Phi)$, where $\Ee$ is a stable vector bundle of rank two on $\PP^3$ with $c_1(\Ee )=c_1$ and $c_2(\Ee )=c_2$, and $\Phi : \Ee \rightarrow \Ee \otimes T_{\PP^3}$ is a non-trivial nilpotent map.
\end{proposition}

\begin{proof}
Since $c_1$ is odd, we get that $c_2$ is even by \cite[Corollary 2.2]{hart}. As in the proof of Proposition \ref{ttt3} we first reduce to the case $c_1=1$ and then use that $\omega _Z\cong \Oo _Z(-3)$, implying $2\chi (\Oo _X)=3c_2$, to exclude the cases $c_2\in \{2,4,6\} $ by the inequality $(c_2-2)(c_2-3)\ge 3c_2$.
\end{proof}


\subsection{Case $\mathrm{Num}(X) \cong \ZZ$}

Now we drop the main assumption on $\mathrm{Pic} (X)$; let $\mathrm{Num}(X)$ be the quotient of $\mathrm{Pic}(X)$ by numerical equivalence. Note that if  $\mathrm{Num}(X) \cong \ZZ$, then the notion of (semi)stability does not depend on the choice of a polarization. For $\Ll \in \mathrm{Pic}(X)$ we call $\deg (\Ll )$ the numerical class of $\Ll$.

\begin{theorem}\label{ii1}
Assume that $\mathrm{Num}(X) \cong \ZZ$ and that $X\ne \PP^n$. If $\Phi : \Ee \rightarrow \Ee \otimes T_X$ is a nilpotent map for a stable reflexive sheaf of rank two on $X$, then we have $\Phi=0$.
\end{theorem}

\begin{proof}
Assume $\Phi \ne 0$ and then $\Ll := \mathrm{ker}(\Phi)$ is a rank one saturated subsheaf of $\Ee$. Set $\Ff := \Ee\otimes \Ll^\vee$. Since $\Ll$ is saturated in $\Ee$, $\Ff$ fits in an exact sequence (\ref{eqb1}) with $r=2$. Since $\Ee$ is stable, we have $\deg (\Aa)>0$ and so $\Aa$ is ample. Call $\Psi: \Ff \rightarrow \Ff\otimes T_X$ the non-zero nilpotent map obtained from $\Phi$.  Since $\Psi \circ \Psi =0$, we have $\Psi (\Ff )\subset u(\Oo _X)\otimes T_X \cong T_X$. Thus $\Psi$ induces a non-zero map $\Aa \rightarrow T_X$. Since $\Aa$ is ample, we have $X= \PP^n$ by \cite{wahl}, a contradiction.
\end{proof}



\section{Arbitrary Picard groups} 

Now we drop the assumption $\mathrm{Pic}(X)\cong \ZZ$, but we fix a very ample line bundle $\Hh \cong\Oo _X(1)$ on $X$ and we use $\Hh$ to check the slope-(semi)stability of sheaves on $X$. We use that $\Oo _X(1)$ is very ample only to guarantee that $\Omega ^1_X(2)$ is spanned. For any torsion-free sheaf of rank two on $X$, define $z_{\Ee}=z_{\Ee, \Hh}$ to be 
$$\mathrm{max}\{ z\in \ZZ~|~ H^0(\Ee(-z)) \text{ has a section not vanishing on a divisor of }X\}.$$
Then we have an exact sequence
\begin{equation}\label{eqc1}
0 \to \Oo _X(z)\to \Ee \to \Ii _Z\otimes \det (\Ee)(-z) \to 0
\end{equation}
with $z =z_{\Ee ,\Hh}$ and $Z\subset X$ of codimension $2$. The integer  $\rho _{2,\Hh}(\Ee)$ is the minimal integer $t$ such that $h^0(\Ii _Z\otimes \det (\Ee)(t-z)) >0$ for
some (\ref{eqc1}). Recall that $x_{\Ee}$ or $x_{\Ee ,\Hh}$ was defined to be the only integer $x$ such that $H^0(\Ee (-x)) \ne 0$ and $H^0(\Ee (-x-1)) =0$. The following result is an adaptation of Lemma \ref{d1}.

\begin{lemma}\label{c1}
Let $\Ee$ be a reflexive sheaf of rank two on $X$. For $a\in \ZZ$ such that
$$a<\mathrm{min}\{ \rho _{2,\Hh}(\Ee ) -z_{\Ee}, \mathrm{max}\{-x_{\Ee ^\vee} -z_{\Ee}, -x_{\Ee}-x_{\det (\Ee)} -z_{\Ee}-1\}\},$$
we have $h^0(\mathcal{E}nd (\Ee)(a)) =h^0(\Oo _X(a))$.
\end{lemma}

\begin{proof}
Since $\Oo _X$ is a factor of $\mathcal{E}nd (\Ee )$, we have $h^0(\mathcal{E}nd (\Ee)(a)) \ge h^0(\Oo _X(a))$ and so it is sufficient to prove the inequality $h^0(\mathcal{E}nd (\Ee )(a)) \le h^0(\Oo _X(a))$.

Set $z:= z_{\Ee ,\Hh}$ and assume that $\Ee$ fits in the exact sequence (\ref{eqc1}) computing the integer  $\rho _{2,\Hh}(\Ee)$. For a fixed $f\in \Hom (\Ee ,\Ee (a))$, let 
$$f_1: \Ee \rightarrow \Ii _Z\otimes \det (\Ee )(-z+a)$$ be the map obtained by composing $f$ with the map $ \Ee(a) \rightarrow \Ii _Z\otimes \det (\Ee )(-z+a)$ twisted from (\ref{eqc1}) with $\Oo _X(a)$. Since $a < \rho _{2,\Hh}(\Ee) -z_{\Ee}$, we have $f_1(\Oo _X(z)) =0$ and so $f$ induces 
$$f_2: \Ii _Z\otimes \det (\Ee )(-z) \rightarrow \Ii _Z\otimes \det (\Ee )(-z+a).$$ 

Now take $g\in H^0(\Oo _X(a))$ inducing $f_2$ and let $\gamma : \Ee \rightarrow \Ee (a)$ be the map obtained by the multiplication by $g$. Then it is enough to prove that $f =\gamma$. Taking $f-\gamma$ instead of $f$, we reduce to the case $g=0$ and in this case we need to prove that $f=0$. From the assumption that $g=0$, we have $f(\Ee )\subseteq \Oo _X(z+a)$, and so $f = 0$ if $-x_{\Ee ^\vee } > z+a$. Note that $\Ee$ is reflexive of rank two and so we have $\Ee ^\vee \cong \Ee\otimes \det (\Ee)^\vee$. Thus $f: \Ee \rightarrow \Oo _X(z+a)$ is induced by a unique $b\in H^0(\Ee (z+a)\otimes \det (\Ee)^\vee)$. If $z+a<-x_{\Ee}-x_{\det (\Ee)^\vee} -1$, we have $b =0$, because $h^0(\Ee (-x_{\Ee}-1)) =0$ and $h^0(\det (\Ee )^\vee(-x_{\det (\Ee)^\vee}-1)) =0$.
\end{proof}

\begin{proposition}\label{c2}
Let $\Ee$ be a reflexive sheaf of rank two on $X\ne\PP^n$ with $\rho _{2,\Hh}(\Ee )\ge 3$ and either $ -x_{\Ee ^\vee}-z_{\Ee}$ or $ -x_{\det (\Ee)}+1 $ at least two. Then any trace-zero co-Higgs field for $\Ee$ is identically zero. 
\end{proposition}

\begin{proof}
Basically the same argument in the proof of Proposition \ref{d2} works with Lemma \ref{d1} replaced by Lemma \ref{c1}. Since $T_X$ is a subsheaf of $\Oo_X(2)^{\oplus N}$ for $N=h^0(\Omega_X^1(2))$, any map  $\Phi: \Ee \rightarrow \Ee \otimes T_X$ induces $N$ elements $\Phi_i : \Ee \rightarrow \Ee(2)$ with $i=1,\ldots, N$. Then by Lemma \ref{c1} each $\Phi_i$ is induced by $f_i\in H^0(\Oo_X(2))$. Now by composing the trace map of $\Phi$ with the inclusion $T_X \subset \Oo_X(2)^{\oplus N}$, we also get $N$ elements $g_i\in H^0(\Oo_X(2))$. We know that $2f_i=g_i$ for each $i$. If $\Phi$ is trace-free, then we get $g_i=0$ and so $f_i=$ for each $i$. Thus $\Phi$ is trivial. 
\end{proof}

\providecommand{\bysame}{\leavevmode\hbox to3em{\hrulefill}\thinspace}
\providecommand{\MR}{\relax\ifhmode\unskip\space\fi MR }
\providecommand{\MRhref}[2]{%
  \href{http://www.ams.org/mathscinet-getitem?mr=#1}{#2}
}
\providecommand{\href}[2]{#2}

\end{document}